\documentclass[a4paper,11pt]{amsart}
\usepackage{amsmath,amssymb,amsthm,amscd,xypic}
\usepackage{enumerate}
\usepackage{color}

\title{Derived equivalences induced by good silting complexes}

\author{Simion Breaz}
\address{Babe\c s--Bolyai University, Faculty of Mathematics and Computer Science \\  1, Mihail Kog\u alniceanu, 400084 Cluj--Napoca, Romania}
\email{bodo@math.ubbcluj.ro}

\author{George Ciprian Modoi}
\address{Babe\c s--Bolyai University, Faculty of Mathematics and Computer Science \\  1, Mihail Kog\u alniceanu, 400084 Cluj--Napoca, Romania}
\email{cmodoi@math.ubbcluj.ro}

\thanks{}


\date{\today}

\newcommand{\la}{\longrightarrow}
\newcommand{\N}{\mathbb{N}}
\newcommand{\Z}{\mathbb{Z}}

\DeclareMathOperator{\Hom}{Hom}
\newcommand{\RHom}{{\mathbb{R}{\Hom}}}
\newcommand{\dotimes}{\otimes^{\mathbb{L}}}
\DeclareMathOperator{\End}{End}
\DeclareMathOperator{\DgEnd}{DgEnd}
\DeclareMathOperator{\Ext}{Ext}
\DeclareMathOperator{\Tor}{Tor}
\DeclareMathOperator{\Ker}{Ker}
\DeclareMathOperator{\Img}{Im}

\newcommand{\Zy}{Z}
\newcommand{\Ho}{H}
\newcommand{\Bd}{B}

\DeclareMathOperator{\codim}{codim}

\newcommand{\A}{\mathcal{A}}

\newcommand{\K}{\mathcal{K}}
\newcommand{\Y}{\mathcal{Y}}
\newcommand{\C}{\mathcal{C}}
\newcommand{\calD}{\mathcal{D}}

\newcommand{\T}{\mathcal{T}}

\newcommand{\U}{\mathcal{U}}
\newcommand{\X}{\mathcal{X}}

\newcommand{\CU}{\mathcal{U}}
\newcommand{\CV}{\mathcal{V}}
\newcommand{\CY}{\mathcal{Y}}

\newcommand{\opp}{^\textit{op}}
\newcommand{\Mod}[1]{\hbox{\rm Mod}({#1})}

\newcommand{\add}[1]{\mathrm{add}({#1})}
\newcommand{\thick}[1]{\mathrm{thick}({#1})}
\newcommand{\Thick}[1]{\mathrm{Thick}({#1})}
\newcommand{\Add}[1]{\mathrm{Add}({#1})}

\newcommand{\Der}[1]{\mathbf{D}({#1})}
\newcommand{\Htp}[1]{\mathbf{K}({#1})}

\theoremstyle{plain}

\newtheorem{thm}{Theorem}[section]
\newtheorem{lem}[thm]{Lemma}
\newtheorem{prop}[thm]{Proposition}
\newtheorem{cor}[thm]{Corollary}

\theoremstyle{definition}

\newtheorem{sett}[thm]{Setting}

\theoremstyle{definition}

\theoremstyle{remark}
\newtheorem{rem}[thm]{Remark}
\newtheorem{expl}[thm]{Example}

\begin{document}

	% ------------------------------------------------------------------------------
	% The body of the paper
\begin{abstract}
We study the equivalences induced by some special silting objects in the derived category over dg-algebra whose positive cohomologies are all zero. 
%Consider a (possibly big) silting object $U$ in a . Under some fairly general appropriate hypotheses,
%		we show that it induces derived equivalences between the derived category over $A$ and a localization of the derived category of dg-endomorphism algebra $B$ of $U$.
%		If, in addition, $U$ is small then this localization is the whole derived category over $B$.
\end{abstract}

	\maketitle

	% ------------------------------------------------------------------------------
	\section*{Introduction}

Tilting theory generalizes Morita theory in the following sense: Let $A$ be a $k$-algebra, where $k$ is a commutative ring with one, and let $T\in\Mod A$ (right modules) be a classical 1-tilting module with endomorphism $E=\End_A(T)$.
Then $T$ induces the torsion theories $(\CU,\CV)$ in $\Mod{A}$ with $\CU=\Ker\Ext_A^1(T,-)$ and $\CV=\Ker\Hom_R(T,-)$, respectively $(\X,\CY)$ in $\Mod{E}$ with $\X=\Ker(-\otimes_E T)$ and $\CY=\Ker\Tor_1^R(-,T)$ such that the functors $\Hom_A(T,-):\CU\rightleftarrows \CY:-\otimes_ET$ and $\Ext_R^1(T,-):\CV\rightleftarrows \X:\Tor_1^R(-,T)$ are mutually inverse equivalences, called \textit{counter equivalences}. We note that in this case $T$ has to be finitely presented. 
A result that makes precise these things was formulated first by Brenner and Butler in \cite{BB} and was afterward called the Tilting Theorem. This result was extended to classical $n$-tilting modules in \cite{My}. 
Note that a version of the Tilting Theorem can be formulated at the level of derived category and derived functors, see \cite{Hap}. This leads to a Morita Theory for derived categories which has as a culminating point the theory developed by Rickard in \cite{Ric1} and \cite{Ric2}. Moreover, a very general result was proved by Keller in \cite{KDG}: every reasonable triangulated category is equivalent to a derived category associated to a dg-algebra. The right choice for this dg-algebra is the dg-endomorphism algebra of a compact generator. This covers the derived version of the tilting theorem since every $n$-tilting module corresponds to a compact generator in the derived category, and the tilting condition assures that the associated dg-endomorphism algebra is the usual endomorphism ring. We refer to \cite{KDT} for a self--contained survey.

%In a first phase, tilting objects which we call now classical were compact (i.e. objects such that the induced Hom-covariant functor commutes with respect to direct sums). 
However, in the study of module categories,  infinitely generated tilting modules play an important role, \cite{GT06}. It was observed (e.g., in \cite{fac1} and \cite{fac2}) that in some cases versions of the above-mentioned counter equivalences  are still valid, but it is not easy to compute the categories of right $E$-modules which are involved. A general result due to Bazzoni in \cite{Baz1} takes a significant step in this sense by introducing so called good $1$-tilting modules. %In the same paper derived aspects of the good tilting modules are also investigated.  
Continuing the same way of thinking, in \cite{BMT} it is shown that if $T\in\Mod{A}$ is a good $n$-tilting module then the derived hom functor $\RHom_A(T,-)$ induces an
equivalence between the derived category $\Der{A}$ and a subcategory of the derived category $\Der{\mathrm{End}_A(T)}$. As in the classical case, these results have a very general version, proved in \cite{NS}.

For practical reasons, the class of classical tilting modules was extended to the class of support $\tau$-tilting modules in \cite{AIR}, which can be interpreted as the compact silting complexes of length $2$ introduced in \cite{KV}. %in order to describe $t$-structures in derived categories. 
It is shown in \cite{HKM} and \cite{Buan} that support $\tau$-tilting modules still induce counter equivalences, but in order to obtain these we have to replace the endomorphism ring of the module with the endomorphism ring of the corresponding compact silting complex as an object in the corresponding derived category. 
 We refer to \cite{FMS}, \cite{Jasso} and \cite{NSZ} for other results on equivalences induced by tilting or compact silting complexes.

As in the case of tilting theory, the theory of support $\tau$-tilting modules was extended to a theory developed for infinitely generated modules, called semi-tilting or silting modules, in \cite{Wei-semi} and \cite{AMV}. Silting modules correspond to non-compact silting complexes of length $2$, more precisely concentrated in $-1$ and $0$ (in fact they are the zeroth cohomology of such a complex). The main ideas from \cite{Baz1} can be applied to these modules/complexes to provide counter equivalences. %induced by silting complexes of length $2$ are described in \cite{BMs}. 
%It would be useful to have a Silting Theorem for non-compact general silting complexes. 
We will apply the general equivalence theorem \cite[Theorem 6.4]{NS} to prove that some special, non-necessarily compact, silting objects over coconnective dg-algebras can be used to construct silting counter equivalences that generalize that from \cite{My}. 
%such a theorem, and we will describe the family of equivalences induced by a silting complex. In order to do this we will use one of the main ideas from \cite{Baz1} by replacing the initial silting complex by a so called good silting complex.
%	The main difference between silting and tilting case is that a silting complex is not quasi-isomorphic (that is not isomorphic in
%	the derived category) to the corresponding silting module, forcing us to consider the differential graded endomorphism algebra instead of the simple endomorphism algebra.

In what follows we outline the organisation of the paper. We start in Section \ref{preli} with some generalities about the dg-algebras and their derived categories. Further, we define t-structures and silting objects in triangulated categories with coproducts, we list some basic properties of them and observe that a dg-algebra is silting in its derived category if and only if it is coconnective (i.e. its positive cohomologies are all zero). %For an object $X$ in a triangulated category we define the (co)resolution dimension of $X$ w.r.t a class of objects $\C$ as being the (minimal) number of steps in which we can construct $X$ out of objects in $\C$ by using (co)cones, and we use them in order to characterize the smallest thick subcategory which contains $\C$.  
    
In Section \ref{der-eq} we consider a dg-$A$-$B$-bimodule $U$ and in Theorem \ref{good-equiv} we give conditions for the derived functor \[\RHom_A(U,-):\Der{A,d}\to\Der{B,d}\] to be fully faithful. We note that in the case when $B$ is exactly the dg-endomorphism algebra of the dg-$A$-module $U$, Theorem \ref{good-equiv} is already contained in \cite[Theorem 6.4]{NS}. Our  proof is based on, seemingly new, Proposition \ref{p:ext-tor}, which is a derived version of the well-known Ext-Tor relations from homological algebra (see \cite[Lemma 1.2.11]{GT06}). Next in the particular case when $B$ is exactly the dg-endomorphism algebra of the dg-$A$-module $U$ we reformulate the hypotheses needed in Theorem \ref{good-equiv} in a more concrete way, by using the thick subcategories generated by $U$.

In the last Section \ref{s:gsc} we consider some special silting objects, called silting complexes, in the derived category associated to a coconnective dg-algebra. In this way, we extend the semi-tilting complexes defined in \cite{Wei-semi}. We prove that these objects are always equivalent to some complexes, called good, that satisfy the hypotheses of Theorem \ref{good-equiv}.  Next, we apply the results in the previous Section in order to obtain some counter equivalences for good silting complexes, similar to the tilting case. We conclude the paper by showing that our results generalize the results involving (derived) equivalences for tilting complexes in \cite{My} and \cite{BMT}.  
    
In this paper all rings/algebras are unital, and $\Mod{A}$ means the category of right $R$-modules, and $\Der{A}$ is the corresponding derived category. If $\C$ is a category then the groups of homomorphisms are denoted by  $\Hom_\C(X,Y)$. If we are inside a derived category, we will use the word isomorphism instead of quasi-isomorphism. We recall that in our cases the hearts associated to $t$-structures are abelian categories, and that two objects in such a heart are isomorphic if and only if they are quasi-isomorphic in the corresponding triangulated category. The notations for categories associated to dg-algebras are explained in the next section. We refer to \cite{A19}, \cite{NS}, and \cite{N} for other unexplained notions.

	\section{Preliminaries}\label{preli}

\subsection{Modules over dg-algebras} %We begin with some generalities about the total derived functors. 
We will work in the context of dg-algebras. We will follow
\cite{KDG} and \cite{KP} in these considerations. Fix a unital and commutative ground ring $k$.
Recall that a $k$-algebra is $\Z$-graded provided that it has a decomposition as a direct sum of $k$-submodules  $A=\bigoplus_{i\in\Z}A^i$ such that $A^iA^j\subseteq A^{i+j}$,
for all $i,j\in\Z$. Further a {\em dg-algebra}  is a $\Z$-graded algebra $A$ endowed with a $k$-linear differential $d:A\to A$ which is homogeneous of degree 1, that is $d(A^i)\subseteq A^{i+1}$ for all $i\in\Z$, and satisfies $d^2=0$ and the graded Leibniz rule:
	\[d(ab)=d(a)b+(-1)^iad(b)\hbox{, for all }a\in A^i\hbox{ and }b\in A.\]
A (right) \textit{dg-module} over $A$ is a $\Z$-graded module $M=\bigoplus_{i\in\Z}M^i$ (i.e., an $A$-module $M$ that admits, as an abelian group, the already mentioned decomposition such that $A^iM^j\subseteq M^{i+j}$) endowed with a $k$-linear square-zero differential $d:M\to M$, which is homogeneous of degree
1 and satisfies the graded Leibnitz rule:
	\[d(xa)=d(x)a+(-1)^ixd(a)\hbox{, for all }x\in M^i\hbox{ and }a\in A.\] 
 Left dg-$A$-modules are defined similarly. A morphism of dg-$A$-modules  is an $A$-linear map $f:M\to N$ such that $f(M^i)\subseteq N^i$ and $f$ commutes with the differential. In this way we obtained the category $\Mod{A,d}$ of all dg-$A$-modules.

If $A$ is a dg-algebra, then the dual dg-algebra $A\opp$ is defined as follows:
$A\opp=A$ as graded $k$-modules, $ab=(-1)^{ij}ba$ for all $a\in A^i$ and all $b\in A^j$ and the same differential $d$. It is clear that a left dg-$A$-module $M$ is a right dg-$A\opp$-module with the ``opposite" multiplication $xa=(-1)^{ij}ax$, for all $a\in A^i$ and all $x\in M^j$, henceforth we denote by $\Mod{A\opp,d}$ the category of left dg-$A$-modules.

It is not hard to see that an ordinary $k$-algebra $A$ can be viewed as a dg-algebra concentrated in degree $0$, case in which a dg-module is nothing else than a complex of ordinary (right) $A$-modules, that is $\Mod{A,d}$ is the category of all complexes (we will not use another special notation for the category of complexes).

\subsection{Canonical Functors}\label{sect:can-functors}
For a dg-module $X\in\Mod{A,d}$ one defines (functorially) the following $k$-modules \[\Zy^n(X)=\Ker(X^n\stackrel{d}\la X^{n+1}),\]
\[\Bd^n(X)=\Img(X^{n-1}\stackrel{d}\la X^{n}),\]
and \[\Ho^n(X)=\Zy^n(X)/\Bd^n(X),\] for all $n\in\Z$. We call $\Ho^n(X)$ the n-th cohomology $k$-module of $X$. A morphism of dg-modules is called \textit{quasi-isomorphism} if it induces
isomorphisms in cohomologies. The dg-module $X\in\Mod{A,d}$ is called \textit{acyclic} if $H^n(X)=0$ for all $n\in\Z$.
A morphism of dg-$A$-modules $f:X\to Y$ is called null--homotopic provided that there is a graded homomorphism $s:X\to Y$ of degree $-1$ such that
$f=sd+ds$.

The homotopy category $\Htp{A,d}$ has the same objects as $\Mod{A,d}$ and the morphisms are equivalence classes of morphims of dg-modules, modulo the homotopy. It is well--known that the homotopy category is triangulated. Moreover, a null--homotopic morphism is acyclic, therefore the functors $\Ho^n$ factor through $\Htp{A,d}$ for all $n\in\Z$.

The derived category $\Der{A,d}$ is obtained from $\Htp{A,d}$, by formally inverting all quasi-isomorphisms.

Let now $A$ and $B$ be two dg-algebras and let $U$ be a dg-$B$-$A$-\textit{bimodule} (that is $U$ is a dg-$B\opp\otimes_kA$-module).
For every $X\in\Mod{A,d}$ then we can consider the so called dg-Hom:
\[\Hom^\bullet_A(U,X)=\bigoplus_{n\in\Z}\Hom_A^n(U,X) \]
with $\Hom_A^n(U,X)=\prod_{i\in\Z}\Hom_{A^0}(U^i,X^{n+i})$, whose differentials are given by
$d(f)(x)=d_Yf(x)-(-1^n)fd_X(x)$ for all $f\in\Hom_A^n(X,Y)$. Then $\Hom^\bullet_A(U,X)$ becomes a dg-$B$-module, so we get a functor
\[\Hom^\bullet_A(U,-):\Mod{A,d}\to\Mod{B,d}.\]
It induces a triangle functor
	\[\Hom^\bullet_A(U,-):\Htp{A,d}\to\Htp{B,d}\]
This functor has a total right derived functor, defined as follows:
A dg-$A$-module $C$ it is called {\em cofibrant} if
$\Hom_{\Htp{A,d}}(C,N)=0$ for all acyclic dg-$A$-module $N$, or equivalently, $\Hom_{\Htp{A,d}}(C,X)=\Hom_{\Der{A,d}}(C,X)$ for every dg-$A$-module $X$. Then a \textit{cofibrant replacement} for $U$ is a cofibrant dg-module $U'$ together with a  quasi-isomorphism $U'\to U$. Dually we define the notions of \textit{fibrant object} and \textit{fibrant replacement}.
It turns out that (co)fibrant replacements always exist in $\Htp{A,d}$ (see \cite[Theorem 3.1]{KDG}), hence we can define
	\[\RHom_A(U,-):\Der{A,d}\to\Der{B,d},\]
by $\RHom_A(U,X)=\Hom^\bullet_A(U',X)$ where $U'$ is a cofibrant replacement of $U$. Note that it can be alternatively defined as $\RHom_A(U,X)=\Hom_A^\bullet(U,X')$ where $X'$ is a fibrant replacement of $X$, since $\Hom^\bullet_A(U',X)\cong\Hom^\bullet_A(U,X')$ in $\Der{B,d}$ (see \cite[Lemma 32.1]{dga} and its dual).

Let $Y\in\Mod{B,d}$. On the usual tensor product $Y\otimes_B U$  there is a natural grading:
\[Y\otimes^\bullet_BU=\bigoplus_{n\in\Z}Y\otimes_B^nU,\]
where $Y\otimes_B^nU$ is the quotient of $\bigoplus_{i\in\Z}Y^i\otimes_{B^0}U^{n-i}$
by the submodule generated by $y\otimes bu-yb\otimes u$ where $y\in Y^i$, $u\in U^{j}$ and $b\in B^{n-i-j}$, for all $i,j\in\Z$. Together with the differential
\[d(y\otimes u)=d(y)\otimes u+(-1)^iyd(u)\hbox{, for all }y\in Y^i, u\in U,\] we get a
dg-$A$-module inducing a functor
$-\otimes^\bullet_BU:\Mod{B,d}\to\Mod{A,d}$ and further a triangle functor
$-\otimes^\bullet_BU:\Htp{B,d}\to\Htp{A,d}.$ The left derived tensor product is defined
by $Y\dotimes_BU=Y'\otimes^\bullet_BU\cong Y\otimes^\bullet_BU'$ where $Y'$ and $U'$ are
cofibrant replacements for $Y$ and $U$ in $\Htp{B,d}$ and $\Htp{B\opp,d}$ respectively.
It induces a triangle functor
\[-\dotimes_BU:\Der{B,d}\to\Der{A,d}\] which is the left adjoint of $\RHom_A(U,-)$.

{We want to note here some basic properties of the derived functors, in order to use them freely whenever we need: 
\begin{enumerate}
\item Since $A$ and $B$ are cofibrant as seen as a dg-$A$-module, respectively dg-$B$-module, the very definition of the derived Hom and tensor functors gives us the isomorphisms: 
\[\RHom_A(A,X)\cong X\hbox{ and }B\dotimes_B U\cong U.\]
 
\item For $X\in\Mod{A,d}$ we have (e.g. see \cite[Section 12.2]{Y20}): \[\Ho^n\RHom(U,X)=\Hom_{\Der{A,d}}(U,X[n]).\] %if $Y\in\Mod{B,d}$, $U\in\Mod{B\opp,d}$ then  $\Ho^n(Y\dotimes_BU)=...$.
%\item $\RHom_A(A,X)\cong X$, $B\dotimes_B U\cong U$.

\item The functor $-\dotimes_BU$ is a left adjoint of $\RHom_A(U,-)$. The standard natural isomorphism \[\Hom_{\Der{A,d}}(-\dotimes_BU,-)\to \Hom_{\Der{B,d}}(-,\RHom_A(U,-))\] induces a natural quasi-isomorphism  \[\RHom_{A}(-\dotimes_BU,-)\to \RHom_{B}(-,\RHom_A(U,-)).\]
\end{enumerate}
}

\subsection{Triangulated categories and silting objects} 

{
Let $\calD$ be a triangulated category with the shift functor denoted by $[1]$, and let $\C$ be a subcategory of $\calD$. For every subset $I\subseteq\Z$ we define the full subcategories of $\calD$
	\[\C^{\perp_I}=\{X\mid \calD(C,X[i])=0\hbox{ for all }i\in I, \textrm{ and all }C\in\C\},\] \[^{\perp_I}\C=\{X\mid \calD(X,C[i])=0\hbox{ for all }i\in I, \textrm{ and all }C\in\C\},\] and $\C[I]=\bigcup_{i\in I}\C[i]$.
If $\C=\C[1]$ (that is, $\C$ is closed under all shifts) then clearly $\C^{\perp_\Z}=\C^{\perp_0}$ and we write simply $\C^\perp$ instead of both.
}

Consider an object $X\in\calD$. Following \cite{Wei-semi}, we say that  $X$ has the {\em $\C$-resolution dimension } (respectively {\em $\C$-coresolution dimension }) $\leq n$, and we write
$\dim_\C X\leq n$, ($\codim_\C X\leq n$)  provided that there  is a sequence of triangles
	\[ X_{i+1}\to C_i\to X_i\overset{+}\to\hbox{ with }0\leq i\leq n \]
\[\hbox{(respectively } X_i\to C_i\to X_{i+1}\overset{+}\to\hbox{ with }0\leq i\leq n\hbox{)}\] in $\calD$, such that
$C_i\in\C$, $X_0=X$ and $X_{n+1}=0$. We will write $\dim_\C X< \infty$ ($\codim_\C X<\infty$) if we can find a positive integer $n$ such that $\dim_\C X\leq n$ (respectively, $\codim_\C X\leq n$).

{ A subcategory $\T$ of $\calD$ is called \textit{thick subcategory} if it is a full triangulated subcategory (i.e., it is closed with respect to extensions and shifts) closed with respect to direct summands. For every class of objects $\C$ of $\calD$ there exists a smallest thick subcategory which contains $\C$, denoted by $\thick \C$.
We also denote by $\Thick{\C}$ the smallest thick subcategory containing $\C$ and also closed under arbitrary direct sums.

\begin{lem}\label{lem:thick}
Let $\calD$ be a triangulated category, $\C$ a set of objects from $\calD$. The following are equivalent for an object $A$:
\begin{enumerate}
    \item $A\in \thick C$;
    \item $A$ is a direct summand of an object $X$ such that $\dim_{\C[\Z]}X<\infty$;
    \item $A$ is a direct summand of an object $X$ such that $\codim_{\C[\Z]}X<\infty$.
\end{enumerate} 
Consequently, $\thick{\Add\C}=\Thick{\C}$. 
\end{lem}

\begin{proof}
  The equivalence between the three conditions follows by  \cite[Lemma 3.16]{NS}.  For the last equality, note that since direct sums of triangles are triangles, it follows immediately that the class of objects $X$ defined by the property  $\dim_{(\Add C)[\Z]}X<\infty$ is closed under direct summands. 
\end{proof}

\begin{cor}\cite[Corollary 3.17]{NS}
If $F:\calD_1\to \calD_2$ is an exact functor between triangulated categories, and $\C$ a set of objects of $\A$. Then $F(\thick \C)\subseteq \thick{F(\C)}.$    
\end{cor}

\begin{expl}\label{compact}
Suppose that $A$ is a dg-algebra. An object $X\in \Der{A,d}$ is compact if and only if $X\in \thick A$, \cite[4.3.4]{NS}.  
\end{expl}
} 

A \textit{t-structure} in the triangulated category $\calD$ is a pair $\mathbf{t}=(\U,\CV)$ of subcategories of $\calD$ such that 
\begin{enumerate}
    \item $\U={^{\perp_0}\CV}$, $\CV=\U^{\perp_0}$,
    \item for every object $X$ from $\calD$ there exists a triangle $U\to X\to V\to U[1]$ such that $U\in\U$ and $V\in \CV$, and 
    \item $\U[1]\subseteq \U$ (i.e., $\U$ is closed under positive shifts).
\end{enumerate}
In this case, the subcategory $\U\cap \CV[1]$ is an abelian category, called \textit{the heart} of $\mathbf{t}$. We recall that the morphisms $U\to X$ and $X\to V$ from (2) are unique up to isomorphism, and they can be used to construct \textit{truncation functors} $u:\calD\to \U$ and $v:\calD\to\CV$, and a cohomological functor $\Ho^0_\mathbf{t}=u(v(-[1])):\calD\to \U\cap \CV[1]$.  Note that under the presence of (1) and (2), the condition (3) is equivalent to $\CV[-1]\subseteq \CV$.
 
If in the definition of t-structures, we replace the condition (3) by $\U[-1]\subseteq \U$, then the pair $(\U,\CV)$ will be a \textit{co-t-structure}, and $\U[1]\cap \CV$ is called \textit{the co-heart} of this co-t-structure. 
 
%\subsection{Silting objects in triangulated categories}  Let $\D$ be a triangulated category. 

An object $U\in \calD$ is called \textit{silting} if the pair $\mathbf{t}_U=(U^{\perp_{>0}},U^{\perp_{\leq 0}})$ is a t-structure. We record from \cite{A19} some basic properties of silting objects:

\begin{prop}\label{prop:basic-silting}
Suppose that $U$ is a silting object in the triangulated category $\calD$. 
\begin{enumerate}[{\rm (1)}]
    \item $\Hom_\calD(U,U[i])=0$ for all $i>0$.
    \item $U^{\perp_{\mathbb{Z}}}=0$ ($U$ is a weak generator for $\calD$).
    \item If $X\in U^{\perp_{>0}}$ then $X$ is a Milnor colimit $V_0\to V_1\to \dots \overset{f_n}\to V_n\to \dots$ such that $V_0$ is a direct sum of copies of $U$ and for every $n>0$ the cone of $f_n$ is a coproduct of copies of $U[n]$.
    \item if $\Hom_\calD(X[-1],U^{\perp_{>0}})=0$ and $X\in U^{\perp_{>0}}$ then $X\in \Add U$.
    \item If $\calD$ has coproducts then $\Ho^0_{\mathbf{t}_U}(U)$ is a projective generator for the heart of $\mathbf{t}_U$.  
    \item If $\calD$ is compactly generated then there exists a class $\X$ such that $(\X,U^{\perp_{>0}})$ is a co-t-structure.
\end{enumerate}
\end{prop}

\subsection{Coconnective dg-algebras}
Let $A=\bigoplus_{i\in\Z}A^i$ be a dg algebra. We say that $A$ is {\em non-positive} (see \cite[2.4]{KY}) if $A^i=0$ for all $i>0$. Moreover, $A$ is a {\em coconnective} a dg algebra if $\Ho^i(A)=0$ for all $i>0$, \cite[Remark 3.3.15]{Y20}. It is observed in \cite[Remark 11.4.39]{Y20} that if $A$ is a coconnective dg-algebra and $\tau^{\leq 0}(A)$ is its smart truncation at $0$ then $\tau^{\leq 0}(A)$ is a non-positive dg-algebra (by \cite[2.5]{KY}), and the embedding $\iota:\tau^{\leq 0}(A)\to A$ is a quasi-isomorphism. Then the derived extension of scalars functor $\iota^\star$, induced by $\iota$, is an equivalence between the corresponding derived categories, \cite[Theorem 12.7.2]{Y20}, whose quasi-inverse is the restriction functor extended to the corresponding derived categories.  %(see also \cite[2.2]{KY}). 

\begin{prop}\label{prop:w-non-pos-silt}
Let $A$ be a dg-algebra. Then $A$ is silting in $\Der{A,d}$ if and only if it is coconnective.    
\end{prop}
\begin{proof}
Suppose that $A$ is silting in $\Der{A,d}$. Then for all $i>0$ we have 
\[0=\Hom_{\Der{A,d}}(A,A[i])=\Ho^i(\RHom_A(A,A))=H^i(A).\]

Conversely, by the above considerations, we can assume that $A$ is non-positive.
There is a (standard) t-structure $(\Der{A,d}^{\leq 0},\Der{A,d}^{>0})$ in $\Der{A,d}$, \cite[Proposition 2.1]{KY}, where \[\Der{A,d}^{\leq 0}=\{X\in \Der{A,d}\mid \Ho^i(X)=0 \textrm{ for all } i> 0\}\] and \[\Der{A,d}^{>0}=\{X\in \Der{A,d}\mid \Ho^i(X)=0 \textrm{ for all }i\leq 0\}.\] It is easy to see, using the property (2) from Section \ref{sect:can-functors}, that $\Ho^i(X)\cong\Hom_{\Der{A,d}}(A,X[i])$, for all $i\in\Z$. Therefore, $\Der{A,d}^{\leq 0}=A^{\perp_{>0}}$ and $\Der{A,d}^{>0}=A^{\perp_{\leq 0}}$.
\end{proof}

\begin{rem} Suppose that $A$ is a coconnective dg algebra. 

1) The heart of the t-structure $(A^{\perp_{>0}},A^{\perp_{\leq 0}})$ consists of those objects $X$ such that $\Ho^i(X)=0$ for all $i\neq 0$, since we also have the equalities $A^{\perp_{>0}}=\Der{A,d}^{\leq0}$ and $A^{\perp_{\leq 0}}=\Der{A,d}^{>0}$. Hence this heart is equivalent to $\Mod{\Ho^0(A)}$.

2) The canonical functors $\Der{A,d}\to A^{\perp_{>0}}$ and $\Der{A,d}\to A^{\perp_{\leq0}}$ correspond to what are called ``smart truncations''  (that exist only in the case of coconnective algebras).

3) Since $\Der{A,d}$ is compactly generated, we also have a co-t-structure $(\X,A^{\perp_{>0}})$. The triangles induced by this (see condition (2) from the definition o co-t-structures), that are no longer functorial, can be interpreted as ``stupid truncations". 
\end{rem}

\begin{rem}\label{notation-coconective} 
Let $A$ be a coconnective dg algebra, and denote $R=H^0(A)$. If $\tau^{\leq 0}(A)$ is the smart truncation of $A$ then we have, besides the morphism $\iota:\tau^{\leq 0}(A)\to A$, also a canonical morphism $p:\tau^{\leq 0}(A)\to R$. We denote by $p^\ast$ and $\iota^\ast$ the derived of the associated induction functors and by  $p_\ast$ and $i_\ast$ the extensions of the restriction functors to the corresponding derived categories. Then  $\iota^\ast$ and $\iota_\ast$ are mutually inverse equivalences since $\iota$ is a quasi-isomorphism. Consequently, $p^\ast\iota_\ast$ is a left adjoint for $\iota^\ast p_\ast$. This procedure, which implies a reduction to the truncated dg-algebra, followed by the induction $p^\ast$ is also used in \cite{BMs} and \cite{XYZ}.   

Although $p^\ast$ and $p_\ast$ are not good enough, \cite{Sa22}, the restrictions of $p^\ast\iota_\ast$ and $\iota^\ast p_\ast$ to the hearts of the standard t-structures of $\Der{A,d}$ and $\Der{R}$ are equivalences of categories. 
For reader's convenience, we note that the subcategory $\overline{\Mod{R}}$ of $\Der{A,d}$ that consists in all objects $X$ with the property $\Ho^i(X)=0$ for all $i\neq 0$ is the heart of the standard t-structure on $\Der{A,d}$. %Therefore, $\overline{\Mod{R}}$ is equivalent to the category $\Mod{R}$ of all (right) $R$-modules.     
\end{rem}

%\textbf{Teoria de cotorsisiune asociata corespunde la ``stupid truncation''?}

%Let $A$ be a non-positive dg-algebra. 

%the subcategory $\thick{\Add A}$ of $\Der{A,d}$ plays an important role, therefore we set a shorthand for it: $\T_A=\thick{\Add A}$. 

We also record the following basic properties that can be proved by using Proposition \ref{prop:w-non-pos-silt}.

\begin{cor} Let $A$ be a coconnective dg-algebra, and $U$ is an object in $\thick{\Add A}$. Then:
\begin{enumerate}[{\rm (a)}]
\item $\Ho^i(U)=0$ for $i\gg0$.
\item $\Ho^i(U)=0$ for all $i>0$ if and only if $\dim_{\Add A}U<\infty$. 
\end{enumerate}
\end{cor}

\begin{proof} %Note that $H^i(U)\cong\Hom_{\Der{A,d}}(A,U[i])$, for all $i\in\Z$ and 
Since $A$ is silting, the subcategory $\Add A$ is a silting subcategory in $\Thick{A}$ in the sense of \cite[Definition 4.1]{A19}. Then (a) follows by \cite[Proposition 4.3]{A19}. 

For (b) observe that if $\Ho^i(U)=0$ for $i>0$ is equivalent to $U\in A^{\perp_{>0}}$, so $\dim_{\Add A}U<\infty$ by \cite[Theorem 4.4]{A19}. Conversely if $\dim_{\Add A}U\leq n$, it follows by induction on  $n$ that $\Ho^i(U)=0$, for $i>0$.
\end{proof}

\section{Derived equivalences}\label{der-eq}

	\subsection{An $\RHom$-$\dotimes$ relation}

{We will present dg analogue for a well-known result in homological algebra, see \cite[Lemma 1.2.11 (c) and (d)]{GT06}. Note that in the original version, the module corresponding to $X$ here has to be injective, but we don't need this assumption more since we replace the usual $\Hom(-,X)$ functor with the derived one, and this functor is exact as a triangulated functor.}
 
\begin{prop}\label{p:ext-tor} Let $A$ and $B$ be two dg-algebras and let $U$ be a dg-$B$-$A$-bimodule. 
\begin{enumerate}[{\rm (a)}]
\item There is a morphism of dg $k$-modules \[\gamma_{X,Y}:\RHom_A(U,X)\dotimes_BY\to\RHom_A(\RHom_{B\opp}(Y,U),X)\] which is natural in both $X\in\Der{A,d}$ and $Y\in\Der{B\opp,d}$.
\item If $Y$ is compact in $\Der{B\opp,d}$ then the natural morphism $\gamma_{X,Y}$ defined above is an isomorphism.
\end{enumerate}
\end{prop}

\begin{proof}
(a) The left dg-$B$-module $Y$ can be regarded as a dg-$B$-$k$-bimodule, so there is a pair of adjoint functors
			\[\RHom_k(Y,-):\Der{k,d}\leftrightarrows\Der{B,d}:-\dotimes_BY.\]
Thus we have an isomorphism
			\[\Gamma_{M,Y,N}:\Hom_{\Der{B,d}}(M,\RHom_k(Y,N))\stackrel{\cong}\to\Hom_{\Der{k,d}}(M\dotimes_BY,N)\] which is natural in $M$, $Y$, and $N$. 
 
By \cite[Lemma 31.4]{dga} we can consider the natural morphism \[H^0:\RHom_A(U,-)\to \Hom_{\Der{A,d}}(U,-).\] 
{ We consider a cofibrant replacement $X'$ of $X$.} It follows that we can view $\RHom(U,X)=\Hom^\bullet_A(U,X')$. For every $f\in \RHom_A(U,X)$ we obtain, in $\Der{A,d}$, a morphism $H^0(f):U\to X$. Applying the functor $\RHom_A(\RHom_{B\opp}(Y,U),-)$ to this morphism, we obtain a morphisms in $\Der{k,d}$ of the form: 
$$\RHom_{A}(\RHom_{{B\opp }}(Y,U),U)\to \RHom_{A}(\RHom_{B\opp}(Y,U),X).$$

The natural morphism  $\mu_Y:Y\to\RHom_A(\RHom_{B\opp}(Y,U),U)$ lies originally in $\Der{B\opp,d}$. By using the restriction of the scalars, we can view it as a morphism from $\Der{k,d}$.
It follows that $\RHom_A(\RHom_{B\opp}(Y,U),H^0(-))\mu_Y$ is a morphism 
\[\RHom_A(U,X)\to \RHom_k(Y,\RHom_A(\RHom_{B\opp}(Y,U),X)).\]  

We define $\gamma_{X,Y}$ as \[\Gamma_{\RHom_A(U,X), Y,\RHom_A(\RHom_{B\opp}(Y,U),X)}(\RHom_A(\RHom_{B\opp}(Y,U),H^0(-))\mu_Y),\]
and the proof is complete.

(b) Since $\mu_B$ is an isomorphism, it can be checked that $\gamma_{X,B}$ is an isomorphism. In fact, $\gamma_{X,B}$ can be identified as the composition of the following natural isomorphisms \[\RHom_A(U,X)\dotimes_BB \overset{\cong}\to \RHom_A(U,X) \overset{\cong}\to \RHom_A(\RHom_{B\opp}(B,U),X).\]
   Moreover, since $\mu_{Y[1]}$ can be identified with $\mu_Y[1]$ for all $Y$, it follows that for all integers $\ell$, the morphisms $\gamma_{X,B[\ell]}$ are isomorphisms.
   
   Consider the full subcategory $\C$ of $\thick{B\opp}$ consisting of all $Y$ for which $\gamma_{X,Y}$ is an isomorphism. As we noted before, $B\in\C$.  Since all involved functors are additive and triangulated it follows that  $\C$ is thick, therefore
   $\thick{B\opp}\subseteq \C$.
   \end{proof}

\subsection{A general derived equivalence}

In the following we will present a part of a very general derived equivalence theorem proved in \cite{NS} that will be used to study silting complexes.

Let $U$ be a dg-$B$-$A$-bimodule. The total derived functors
\[\RHom_A(U,-):\Der{A,d}\leftrightarrows\Der{B,d}:-\dotimes_BU\]
form an adjoint pair.  Moreover, the total derived functors:
\[\RHom_A(-,U):\Der{A,d}\leftrightarrows\Der{B\opp,d}:\RHom_{B\opp}(-,U)\] form a right
adjoint pair (see \cite[Lemma 13.6]{KDU}). The associated natural morphisms will be denoted by:
\[\phi:\RHom_A(U,-)\dotimes_BU\to1_{\Der{A,d}}\hbox{ and }\] \[\psi:1_{\Der{B,d}}\to\RHom_A(U,-\dotimes_BU),\] respectively
\[\delta:1_{\Der{A,d}}\to\RHom_{B\opp}(\RHom_A(-,U),U)\hbox{ and }\] \[\mu:1_{\Der{B\opp,d}}\to\RHom_A(\RHom_{B\opp}(-,U),U).\]

%{\color{magenta} In \cite{NS} dg-balanced modules are called homologically faithfull balanced.}

%{\color{magenta}Conditia $\dim_{\add B} U<\infty$ nu e cumva echivalent\u a cu faptul c\u a $U$ e compact peste $B$? Cred ca e $B^{op}$ in loc de B!!!}

\begin{thm}\label{good-equiv} Let $U$ be a dg-$B$-$A$-bimodule such that $U$ is compact in $\Der{B\opp,d}$ and $\delta_A$ is an isomorphism. 
\begin{enumerate}[{\rm (1)}]
\item The functor
$\RHom_A(U,-):\Der{A,d}\to\Der{B,d}$ is fully faithful. 

\item We have an equivalence of categories
\[\RHom_A(U,-):\Der{A,d}\stackrel{\sim}\rightleftarrows \Ker(-\dotimes_BU)^\perp:-\dotimes_BU.\]

\item If, in addition, { $U$ is compact in $\Der{A,d}$ too, and $\mu_B$ is an isomorphism}, then $\Ker(-\dotimes_BU)=0$, hence the categories $\Der{A,d}$ and $\Der{B,d}$ are equivalent.
\end{enumerate}
\end{thm}

\begin{proof}
(1) The statement is equivalent to the fact that the counit $\phi$ of the adjunction between $\RHom_A(U,-)$ and $U\dotimes_B-$ is an isomorphism.

Let $X\in\Der{A,d}$.  Since $U$ is compact in $\Der{B\opp,d}$, the natural morphism $\gamma_{X,U}$ constructed in  Proposition \ref{p:ext-tor} is an isomorphism. Then the composite natural morphism
\begin{align*}\RHom_A(U,X)\dotimes_BU\overset{\gamma_{X,U}}\longrightarrow&\RHom_A(\RHom_{B\opp}(U,U),X)\\ \overset{\cong}\longrightarrow&\RHom_A(A,X)\overset{\cong}\longrightarrow X,\end{align*}
is an isomorphism too. { Using the definition of $\gamma_{X,U}$ from the proof of Proposition \ref{p:ext-tor}, it follows that this composition can be identified with $\Gamma_{\RHom_A(U,X),U,X}(1_{\RHom_A(U,X)})=\phi_X$.} 
 
% The equivalence between the statements that the counit is a natural isomorphism and that the right adjoint is fully faithful is well known.

(2) It is clear that the fully faithful functor $\RHom_A(U,-)$ induces an equivalence between $\Der{A,d}$ and its essential image. Since the functor $\RHom_A(U,-)$ is the right adjoint of the functor $-\dotimes_BU$, it follows that \[\RHom_A(U,\Der{A,d})\subseteq \Ker(-\dotimes_BU)^\perp.\]
For the converse inclusion, the adjunction assures that for any $Y\in\Der{B,d}$ we have $(\psi_Y\dotimes_BU)\phi_{Y\dotimes_BU}=1_{\RHom_A(U,Y\dotimes_BU)\dotimes_BU}$. But $\phi_{Y\dotimes_BU}$ is known to be an isomorphism by the first part of this proof, hence
$\psi_Y\dotimes_BU$ is an isomorphism too. Therefore, if we complete $\psi_Y$ to a triangle
$$Z\to Y\overset{\psi_Y}\la \RHom_A(U,Y\dotimes_BU)\to Z[1],$$
we obtain $Z\dotimes_BU=0=Z[1]\dotimes_BU$, that is $Z,Z[1]\in\Ker(-\dotimes_BU)$.

If further we suppose $Y\in \Ker(-\dotimes_BU)^\perp$, then $\RHom_A(U,Y\dotimes_BU)\cong Y\oplus Z[1]$. But $\RHom_A(U,Y\dotimes_BU)\in\Ker(-\dotimes_BU)^\perp$, thus $Z[1]=0$, so $Y\cong\RHom_A(U,Y\dotimes_BU)$ is in the essential image of $\RHom_A(U,-)$.

(3) { Suppose that $U$ is also compact in $\Der{A,d}$ and that $\mu_B$ is an isomorphism. It follows \[B\cong\RHom_A(\RHom_{B\opp}(B,U),U)\cong\RHom_A(U,U).\]
Moreover, since $U\in\thick A$, we get 
\[B\in\RHom_A(\thick A,U)\subseteq\thick{\RHom_A(A,U)}=\thick U.\] 
Thus, if $Y\dotimes_BU=0$ then $Y\dotimes_BW=0$ for all $W\in\thick U$. In particular $Y\dotimes_BB=0$, hence $Y=0$, proving that $\Ker(-\dotimes_BU)=0$. }
\end{proof}

It is clear that in the above theorem $\Ker(-\dotimes_BU)$ is a localizing subcategory of $\Der{B,d}$. A direct application of Bousfield localization theory presented in \cite[Chapter 9]{N} gives us an equivalence of categories \[\Ker(-\dotimes_BU)^\perp\stackrel{\sim}\la\Der{B,d}/\Ker(-\dotimes_BU).\] Therefore, we have the following

\begin{cor}
Let $U$ be a dg-$B$-$A$-bimodule such that $U$ is compact in $\Der{B\opp,d}$ and $\delta_A$ is an isomorphism. Then there is an equivalence of categories 
$\Der{A,d}\stackrel{\sim}\la\Der{B,d}/\Ker(-\dotimes_BU).$ 
\end{cor}

%\begin{proof} It is clear that $\Ker(-\dotimes_BU)$ is a localizing subcategory of $\Der{B,d}$, hence a direct application of Bousfield localization theory presented in \cite[Chapter 9]{N} gives us an equivalence of categories \[\Ker(-\dotimes_BU)^\perp\stackrel{\sim}\la\Der{B,d}/\Ker(-\dotimes_BU).\]
%	Composing it with the equivalence given in Theorem \ref{good-equiv} we obtain the conclusion. 
	%Further the equivalence \[\Der{B,d}/\Ker(-\dotimes_BU)
	%\stackrel{\sim}\la
	%\Ker(-\dotimes_BU)^\perp\] follows by standard arguments arguments used in Bousfield localization theory (see \cite[Chapter 9]{N}).
%\end{proof}

If $B$ is the dg-endomorphism algebra $B=\DgEnd_A(U)$ of $U$, the hypothesis of Theorem \ref{good-equiv} can be replaced by a property of $U$ in $\Der{A,d}$.

\begin{prop}\label{p:thick-delta} Assume that $U\in \Der{A,d}$ and denote $B=\DgEnd_A(U)$. The following are equivalent:
\begin{itemize}
\item[a)] $A\in\thick U$;
\item[b)] $U\in\thick{B},$ and $\delta_A:A\to\RHom_{B\opp}(U,U)$ is an isomorphism in $\Der{A,d}$.
\end{itemize}
%Then {\rm a)}$\Rightarrow${\rm b)}. If , then {\rm b)}$\Rightarrow${\rm a)}. 
\end{prop}

\begin{proof} a)$\Rightarrow$b) 
For $A\in\thick U$, we get \begin{align*} U\cong\RHom_A(A,U)&\in\RHom_A(\thick U,U)\\ &\subseteq\thick{\RHom_A(U,U)}=\thick B.\end{align*}
 Moreover $\delta_U$ is the isomorphism (since $U\cong\RHom_{B\opp}(B,U)$), so $\delta_X$ is an isomorphism for all $X\in\thick U$, in particular $\delta_A$ is an isomorphism.

b)$\Rightarrow$a) If $U\in\thick B$ and $\delta_A$ is an isomorphism then  
\begin{align*} A\cong\RHom_{B\opp}(U,U)&\in\RHom_{B\opp}(\thick B,U)\\ &\subseteq\thick{\RHom_{B\opp}(B,U)}=\thick U.\end{align*}
\end{proof}

Recall that in many classical equivalence theorems, the existence of an equivalence between categories associated to right modules implies the existence of a similar equivalence for the categories associated to left modules. For instance, this phenomenon is present for  Morita equivalent rings or for derived equivalences induced by classical $n$-tilting modules, \cite[Theorem 1.5]{My} (see also the \cite[Section 1]{NS}).

\begin{cor}\label{cor:right-left-equiv}
Suppose that $U$ is a dg-$A$-module and $B=\DgEnd_A(U)$. If 
\begin{itemize}
\item[i)] $A\in\thick U$, and
\item[ii)] $B\in\thick{U}$
\end{itemize}
then the functors \[\RHom_A(U,-):\Der{A,d}\rightleftarrows\Der{B,d}:-\dotimes_B U\] and \[\RHom_{B\opp}(U,-):\Der{B\opp,d}\rightleftarrows\Der{A\opp,d}:U\dotimes_A-\] are equivalences of categories.
\end{cor}

\begin{rem}\label{rem:B=dg-endo}
Since $B=\DgEnd_A(U)$, the condition ii) can be replaced by $U\in \thick A$.
\end{rem}

\subsection{Complexes of finite (co)dimension}

For the reader's convenience, we present a more detailed result with the same flavor as Proposition \ref{p:thick-delta}.

\begin{prop}\label{l:delta-qis} If $U$ is a dg-$A$-module and $B=\DgEnd_A(U)$, the following are equivalent:
\begin{itemize}
\item[a)] $\codim_{\add U}A\leq n$;
\item[b)] $\dim_{\add B} U\leq n$ and	$\delta_A:A\to\RHom_{B\opp}(U,U)$ is an isomorphism in $\Der{A,d}$.
\end{itemize}
		\end{prop}

\begin{proof} a)$\Rightarrow$b) %The first statement is a consequence of Lemma \ref{dim-codim}.
There exists a sequence of triangles in $\Der{A,d}$
		\[(\dagger)\ \ \  A_i\to U_i\to A_{i+1}\overset{+}\to\hbox{ with }0\leq i\leq n\]
		such that  $U_i\in\add U$, $A_0=A$ and $A_{n+1}=0$. Applying the exact functor $\RHom_A(-,U)$ we get triangles in $\Der{B\opp,d}$ of the form
		\[(\ddagger)\ \ \  V_{i+1}\to B_i\to V_i\overset{+}\to\hbox{ with }0\leq i\leq n \]
		with $B_i=\RHom_A(U_i,U)$ and $V_i=\RHom_A(A_i,U)$, for $0\leq i\leq n$.
		
Observing that $V_0=\RHom_A(A,U)\cong U$, $V_{n+1}=0$ and %that by Lemma \ref{add} and 
$\RHom_A(U,U)\cong B$, it follows that $B_i\in\add B$ for all $0\leq i\leq n$, we obtain $\dim_{\add B} U\leq n$.

Compare the triangles  $(\dagger)$ with the corresponding triangles obtained by applying the functor $\RHom_{B\opp}(-,U)$ to $(\ddagger)$:
\[\diagram
A_i\rto\dto^{\delta_{A_i}}  &U_i\rto\dto^{\delta_{U_i}}  & A_{i+1}\rto^{+}\dto^{\delta_{A_{i+1}}}   &\\
\RHom_{B\opp}(V_i,U)\rto&\RHom_{B\opp}(B_i,U)\rto&\RHom_{B\opp}(V_{i+1},U)\rto^{\ \ \ \ \ \ \ +}&
\enddiagram.\]
It is clear that $\delta_U$ is the isomorphism $U\cong\RHom_{B\opp}(B,U)$, therefore $\delta_{U_i}$, $0\leq i\leq n$ are isomorphisms too. % by Lemma \ref{add}.
Moreover $A_{n+1}=0=\RHom_A(V_{n+1},U)$, hence we deduce $\delta_{A_i}$ are isomorphisms in $\Der{A,d}$ for all $0\leq i\leq n$. For $i=0$ we obtain exactly our conclusion, namely that $\delta_A$ is a (quasi-)isomorphism.

b)$\Rightarrow$a) It is enough to start with a sequence of triangles $(\ddagger)$ as before, and to apply the functor $\RHom_{B\opp}(-,U)$. 
\end{proof}

\begin{rem}
 Using a similar technique it also follows that if $\dim_{\add A} U\leq m$ for some $m\in\N$, then $\codim_{\add U} B\leq m$.   
\end{rem}

\section{Good silting complexes}\label{s:gsc}

In this section, we will apply Theorem \ref{good-equiv} to obtain details about the equivalences induced by silting complexes over coconnective dg-algebras.

\subsection{Silting complexes over coconnective dg-algebras} 	

Let $A$ be a coconnective dg-algebra. From \cite[Example 4.7(1)]{A19} and Remark \ref{rem:B=dg-endo} it follows that every compact silting object $U$ from $\Der{A,d}$ satisfies the hypothesis of Theorem \ref{good-equiv} since $\thick{U}=\thick{A}$.   

In what follows, we will study equivalences that are induced by some silting objects that generalize the compact ones. % in the same strategy as that approached by Wei in \cite{Wei-semi}. 
%If $X$ is an object from $\Der{A,d}$, we use the notation $\Thick{X}=\thick{\Add{X}}$. 
Following the same terminology as that used in \cite[Section 5]{A19} and \cite{Wei-semi} for the case of ordinary rings, we will say that a dg-$A$-module is a {\em silting complex} if $\Add{U}\subseteq U^{\perp_{>0}}$ and $\Thick{U}=\Thick A$. Note that for every silting complex, all its shifts are also silting complexes. It follows that we can assume w.l.o.g. that $U\in\Der{A,d}^{\leq 0}$. Using \cite{Br23}, cosilting complexes can be characterized as in the following:

\begin{thm}\label{thm:bounded-silting} 
Suppose that $A$ is a coconnective dg-algebra.
The following are equivalent for a object $U\in\Der{A,d}^{\leq 0}$:
\begin{enumerate}[{\rm (1)}]
    \item $\Add U\subseteq U^{\perp_{>0}}$ and $\Thick A=\Thick U$ (i.e., $U$ is a silting complex);
    \item $U$ is silting and there exists a positive integer $n$ such that 
    \[\Der{A,d}^{\leq -n}\subseteq U^{\perp_{>0}}\subseteq \Der{A,d}^{\leq 0};\]
    \item \begin{itemize}
    \item[(S1)] $\dim_{\Add A}U<\infty$,
    \item[(S2)] $\Add U\subseteq U^{\perp_{>0}}$,
    \item[(S3)] $\codim_{\Add U}A<\infty$.
    \end{itemize}
\end{enumerate}
\end{thm}

If in Theorem \ref{thm:bounded-silting} $n$ is the smallest positive integer such that the condition (1) is satisfied then we will say that $U$ is an \textit{$n$-silting complex}. If $U$ verifies only the conditions (S1) and (S2) then it is called \textit{presilting}. The reader can find more properties of these objects in \cite{ZW}.

\begin{cor}\label{c:bounded-silting}
A compact dg-module $U\in\Der{A,d}^{\leq 0}$ is a silting complex if and only if the following conditions are satisfied:
\item \begin{itemize}
    \item[(s1)] $\dim_{\add A}U<\infty$,
    \item[(s2)] $U\in  U^{\perp_{>0}}$,
    \item[(s3)] $\codim_{\add U}A<\infty$.
    \end{itemize}
\end{cor}

\begin{proof}
Suppose that $U$ is silting. From \cite[Example 4.7]{A19} it follows that $U$ is silting in $\thick A$. The conclusion (s1) follows from \cite[Lemma 3.8]{A19}, \cite[Proposition 2.17]{AiI} and \cite[Proposition 2.7]{IyY}. The same results, combined with the fact that there exists a positive integer $n$ such that $A[n]\in U^{\perp_>0}$, can be used to obtain (s3). 

The converse implication follows from Theorem \ref{thm:bounded-silting}.
\end{proof}

Therefore, using Corollary \ref{cor:right-left-equiv}, since $A\opp$ is silting, we can state a generalization of \cite[Proposition 1.6]{My}:

\begin{cor}
If $U$ is a compact silting complex as a right dg-module over the coconnective dg-algebra $A$ and $B$ is the dg-endomophism algebra of $U$, then $U=U\dotimes_A A\opp$ is also a compact silting complex as a left dg-module over $B$.     
\end{cor}

%\begin{proof}
%In order to obtain the properties (s1) and (s3) we apply the derived functor $\RHom_{\Der{A,d}}(-,U)$ on the triangles provided by the properties (s3), respectively (s1), of the right dg-module $U$.      
%\end{proof}

\subsection{Equivalences induced by good silting complexes}

A silting complex is called {\em good} if the condition (S3) above can be replaced by the strongest condition (s3).
Two silting complexes $U$ and $U'$ are {\em equivalent} if	$\Add U=\Add{U'}$ or, equivalently, if they induce the same t-structure (see \cite[Definition 4.6]{PV}).

\begin{lem} If $U$ is a silting complex then there is a good silting complex $U'$, such that $U$ and $U'$ are equivalent. 
\end{lem}

\begin{proof}
Assume that $\codim_{\Add U} A\leq n$. There are triangles
\[A_i\to U_i\to A_{i+1}\overset{+}\to\hbox{ with }0\leq i\leq n\]
such that  $U_i\in\Add U$, $A_0=A$ and $A_{n+1}=0$. %By \cite[Proposition 3.9]{Wei-semi},
If $U'=U\oplus \left(\bigoplus^n_{i=0}U_i\right)$ then $\Add{U'}=\Add U$. Hence $U'$ is a silting complex that is equivalent to $U$. It is clear that
$U'$ is good since all $U_i$ are direct summand of $U'$.
\end{proof}

In the following, we will use the following
\begin{sett}\label{Set:coconnective-silt}
Assume that $A$ is a coconnective dg-algebra, and denote $R=\Ho^0(A)$. Let $\overline{\Mod{R}}$ be the subcategory of $\Der{A,d}$ of all objects $X$ with the property $\Ho^i(X)=0$ for all $i\neq 0$ i.e., the heart of the standard t-structure on $\Der{A,d}$. Therefore, $\overline{\Mod{R}}$ is equivalent to the category $\Mod{R}$ of all (right) $R$-modules.  

We also consider a good silting complex $U\in\Der{A,d}$. We will denote the dg-endomorphism algebra of $U$ by $B=\DgEnd_A(U)$, and the endomorphism ring of $U\in\Der{A,d}$ will be $E=\Hom_{\Der{A,d}}(U,U)=\Ho^0(B)$. %Since $E=\Ho^0(B)$, we have a canonical morphism $p:B\to E$. We will denote by $p^\ast$ and $p_\ast$ the derived induction and restriction functors.
\end{sett}

According to Proposition \ref{l:delta-qis}, $U$ is compact as a dg-$B\opp$-module. We can apply Theorem \ref{good-equiv} to obtain an equivalence of categories
		\[\RHom_A(U,-):\Der{A,d}\leftrightarrows\K^\perp:-\dotimes_BU,\]
where $\K=\Ker(-\dotimes_BU)$.

Consequently, we have the following version of \cite[Lemma 31.4]{dga}.

\begin{lem}\label{lem:dga31.4} 
Assume that we are in the Setting \ref{Set:coconnective-silt}. Then for every dg-$A$-module $X$ we have an isomorphism of $E$-modules
$H^i(\RHom_A(U,X))\cong \Hom_{\Der{A,d}}(U,X[i])$ that is functorial in $X$.
\end{lem}

\begin{proof}
This follows from the isomorphisms:
\begin{align*}
H^i(\RHom_A(U,X))&\cong H^i(\RHom_B(B,\RHom_A(U,X)))\\ &\cong \Hom_{\Der{B,d}}(B,\RHom_A(U,X)[i])\\
& \cong \Hom_{\Der{B,d}}(\RHom_A(U,U),\RHom_A(U,X[i])) \\
& \cong \Hom_{\Der{A,d}}(U,X[i]).
\end{align*}
\end{proof}

Hence, we obtain the following basic properties:

\begin{lem}\label{lem:basic-prop-silt}
Using the notations from Setting \ref{Set:coconnective-silt}, the following are true:
\begin{enumerate}[{\rm (a)}]
\item The dg-algebra $B$ is coconnective.
\item The heart, $\overline{\Mod{E}}$, of the standard t-structure on $\Der{B,d}$ is equivalent to the category $\Mod E$;
\item If $\mathcal{H}_U$ is the heart of the t-structure induced by $U$ then $\RHom_A(U,\mathcal{H}_U)=\overline{\Mod E}\cap \K^\perp$.
\item Using the notations and identifications from Remark \ref{notation-coconective}, we have the equality $\Ho^0(\overline{\Mod{E}}\cap \K^\perp)=\Mod{E}\cap p^\ast\iota_\ast(\K^\perp).$ %where $p^\ast\iota_\ast:\Der{B,d}\to \Der{E}$ is the induction functor associated to the canonical morphism $B\to E$.  
\end{enumerate}
\end{lem}

\begin{proof}
The statements (a), (b), and (c) can be obtained by using the equivalence induced by $U$ and Lemma \ref{lem:dga31.4}. The proof for (d) can be obtained by adapting that of \cite[Lemma 4.3.2]{BMs}. 
\end{proof}

%We denote by $q_\ast$ some quasi-inverses for the $H^0$ functors. 
We will see $\Mod R$ and $\Mod E$ as subcategories in the corresponding derived categories $\Der{R}$, respectively $\Der{E}$. If $i$ is an integer, we denote
		\[\X_i=\{X\in\Mod R\mid \Hom_{\Der{R}}(U,X)[j])=0\hbox{ for all }j\neq i\},\]
and	\[\Y_i=\{Y\in\Mod E\mid \Ho^j(\iota^\ast p_\ast(Y)\dotimes_BU)=0\hbox{ for all }j\neq -i\}.\]

%
%		\[\X_i=\{X\in \overline{\Mod{R}}\mid \Ho^j(\RHom_A(U,X))=0\hbox{ for all }j\neq i\},\]
%		\[\Y_i=\{Y\in\overline{\Mod{E}}\mid \Ho^j(Y\dotimes_BU)=0\hbox{ for all }j\neq -i\}.\]

\begin{prop}\label{prop:counter-equiv-general}
Let $i$ be an integer. Then we have an equivalence of categories $\Hom_{\Der R}(U,(-)[i]):\X_i\to \Y_i\cap p^\ast\iota_\ast(\K^\perp),$ whose quasi-inverse is $H^{-i}(\iota^\ast p_\ast (-)\dotimes_{B}U)$. %where $\iota^\ast p_\ast:\Der{E}\to \Der{B,d}$ is the derived restriction functor associated to the morphism $B\to E$. 
%
%$\RHom_A(U,-)$ and $-\dotimes_B U$ induce the equivalence:
%\[\RHom_A(U,-)[i]:\overline{\X}_i\leftrightarrows\overline{\Y}_i\cap\K^\perp:(-\dotimes_BU)[-i],\]
%where $\K=\Ker(-\dotimes_BU)$.
\end{prop}

\begin{proof}
We consider the classes 
\[\overline{\X}_i=\{X\in \overline{\Mod{R}}\mid \Ho^j(\RHom_A(U,X))=0\hbox{ for all }j\neq i\},\]
and		\[\overline{\Y}_i=\{Y\in\overline{\Mod{E}}\mid \Ho^j(Y\dotimes_BU)=0\hbox{ for all }j\neq -i\}.\]

For $X\in\overline{\X}_i$ we put $Y=\RHom_A(U,X[i])$.
Then for all $j\neq0$ we have  
\begin{align*}
\Ho^{j}(Y)&=\Ho^j(\RHom_A(U,X[i]))\cong\Ho^j(\RHom_A(U,X)[i])\\  &
			\cong\Ho^{j+i}(\RHom_A(U,X))=0,
		\end{align*} 
hence $Y\in\overline{\Mod{E}}$.

Since the natural map $\RHom_A(U,X[i])\dotimes_BU\to X[i]$ is a quasi-isomorphism, for all $j\neq -i$ we have
\[\Ho^j(Y\dotimes_BU)=\Ho^j(\RHom_A(U,X[i])\dotimes_BU)\cong\Ho^j(X[i])=\Ho^{j+i}(X)=0,\] hence  $Y\in\overline\Y_i$.  
On the other hand $Y=\RHom_A(U,X[i])\in\K^\perp$, thus $Y\in\overline\Y_i\cap\K^\perp$.

Conversely, if $Y\in\overline\Y_i\cap\K^\perp$, we consider the object $X=Y[-i]\dotimes_BU$. For all $j\neq0$ we have
\[\Ho^j(X)=\Ho^j(Y[i]\dotimes_BU)\cong\Ho^j((Y\dotimes_BU)[i])=\Ho^{j-i}(Y\dotimes_BU)=0,\] hence $X\in\overline{\Mod{R}}$. 

Moreover, since $Y\in\K^\perp$ and $\K^\perp$ is closed under all shifts, we obtain $Y[-i]\in\K^\perp$. It follows that the natural map
$Y[-i]\to\RHom_A(U,Y[-i]\dotimes_BU)$ is a quasi-isomorphism. Since for all $j\neq i$ we have
		\begin{align*}\Ho^j(\RHom_A(U,X))&=\Ho^j(\RHom_A(U,Y[-i]\dotimes_BU)\\ &\cong\Ho^j(Y[-i])=\Ho^{j-i}(Y)=0,
		\end{align*}
we conclude that $X\in\overline{\X}
_i$.

It follows that the functors $\RHom_A(U,-)$ and $-\dotimes_B U$ induce the equivalence:
\[\RHom_A(U,-)[i]:\overline{\X}_i\leftrightarrows\overline{\Y}_i\cap\K^\perp:(-\dotimes_BU)[-i].\]
Using Lemma \ref{lem:basic-prop-silt} and Lemma \ref{lem:dga31.4}, we obtain the conclusion by applying the equivalences described in Remark \ref{notation-coconective}. 
%
%		If $U$ small tilting, then $\Y_i\cap\K^\perp=\Y_i\cap\Der{B,d}=\Y_i$.
	\end{proof}

\begin{rem}
If $\dim_{\Add{A}}U\leq n$, Proposition \ref{prop:counter-equiv-general} gives us non-trivial information only if $0\leq i\leq n$ since $\Ho^i(\RHom_A(U,\overline{\Mod{R}}))=0$ for all $i\notin [0,n]$. Since $U^{\perp_\Z}=0$, we conclude that for all $i\notin [0,n]$ we have $\overline{\X}_i=0$.      
\end{rem}

\begin{rem}
If $U$ is compact in $\Der{A,d}$, then $\K^\perp=\Der{B,d}$, hence Proposition \ref{prop:counter-equiv-general} provides us equivalences between the classes $\X_i$ and $\Y_i$. 
\end{rem}

Proposition \ref{prop:counter-equiv-general} can be applied to provide some counter equivalences already described in the literature. For instance, it can be used to $2$-silting complexes over classical rings to obtain \cite[Theorem 4.3.1]{BMs} for non-compact silting complexes of length $2$, respectively \cite[Theorem 1.1]{Buan} and \cite[Theorem 2.15]{HKM} for the compact case (see also \cite{XYZ}). Of course, all these generalize the classical Brenner-Butler Theorem.  In the following we will give some details for the case on $n$-tilting modules, in order to obtain the Tilting Theorems proved in \cite{BMT} and \cite{My}.

\subsection{Tilting equivalences in module categories}
 Let $A$ be an ordinary $k$-algebra. As mentioned, $A$ may be seen as a dg-algebra concentrated in degree 0. 
 As usual, $\Mod A$ is the category of ordinary (right) modules over $A$.
We keep the notation $\Mod{A,d}$ for the category of complexes of $A$-modules but we denote simply $\Htp A$ and $\Der A$ the respective homotopy and derived category.

An ($n$-)silting complex $U\in\Der{A}$ is called {\em ($n$-)tilting} if, instead condition
$(S2)$, it satisfies the strongest condition $U^{(I)}\in U^{\perp_{\neq0}}$ for all sets $I$.

	%Denote by $(\D^{\leq0}, \D^{\geq0})$ the standard t-structure in $\Der{A,d}$ and by %$\M=\D^{\leq0}\cap\D^{\geq0}$ its heart. Then $\M$ is an abelian category ...

According to \cite[Corollary 3.6 and Corollary 3.7]{Wei-semi} $U\in\Der{A}$ is tilting exactly if  it is isomorphic to its 0-th cohomology, that is to $T=\Ho^0(U)\in\Mod{A}$ (actually $U$ is a projective resolution of $T$), and the $A$-module $T$
satisfies the usual properties defining a tilting module, namely
\begin{enumerate}[T1.]
		\item $T$ is of finite projective dimension.
		\item $\Ext^k_A(T,T^{(I)})=0$ for all sets $I$ and all positive integers $k$.
		\item There is an exact sequence $0\to A\to T_0\to\ldots\to T_n\to 0$ with $T_i\in\Add T$ for all $0\leq i\leq n$.
\end{enumerate}
Moreover, in this case $\Hom_{\Der{A}}(U,U)\cong\Hom_{\Der{A}}(T,T)=\Hom_A(T,T)$. We note that the tilting module $T$ is good in the sense of \cite{BMT} exactly if the corresponding tilting object $U\in\Der{A}$ is good in the sense used in the present paper. A tilting module $T$ in $\Mod{A}$ is called {\it classical}, provided that it has a projective resolution $U$ such that all its entries are finitely generated projective modules i.e., $U$ is a compact tilting object in $\Der{A}$.

\begin{thm}\label{BMT}\cite[Theorem 2.2 and Corollary 2.4]{BMT}.
Let $T\in\Mod{A}$ be a good tilting module and let $E=\End_A(T)$. Then $T$ is a $E$-$A$-bimodule and there is an isomorphism of $k$-algebras $A\cong\End_{E\opp}(T)$. 

Further, the adjunction morphism $\phi:\RHom_A(T,X)\dotimes_AU\to X$ is an isomorphism for all $X\in\Der A$, hence the functor	\[\RHom_A(T,-):\Der{A}\to\Der{E}\]  is fully faithful. Consequently, it induces equivalences of categories
		\[\Der{A}\stackrel{\sim}\la\Der{E}/\K\stackrel{\sim}\la
		\K^\perp,\] where $\K=\Ker(-\dotimes_BU)$. 
  
Moreover $\K^\perp=\Der{E}$, provided that $T$ is classical tilting.
\end{thm}

\begin{proof}
In order to apply the above results, denote by $U$ a projective resolution of $T=\Ho^0(U)$, that is $U\in\Der{A}$ is a good tilting object. Then
\[\Ho^n(\DgEnd_A(U))=\Ho^n(\RHom_A(U,U))\cong\Hom_{\Der{A}}(U,U[n])=0\] for $n\neq0$, hence $\DgEnd_A(U)$
is concentrated in degree 0. Since
\[\Ho^0(\DgEnd_A(U))=\Hom_{\Der{A}}(U,U)\cong\Hom_A(T,T)=E,\]  the quasi-isomorphism $A\to\RHom_{B\opp}(U,U)$ of Theorem \ref{good-equiv} becomes an isomorphism of $k$-algebras $A\to\End_{E\opp}(T)$.

Finally \[\Mod{\DgEnd_A(U),d}=\Mod{E,d}\] is the category of differential complexes of $E$-modules,
\[\Htp{\DgEnd_A(U),d}=\Htp B \hbox{ and }\Der{\DgEnd_A(U),d}=\Der E,\] and the conclusion follows by Theorem \ref{good-equiv}.
\end{proof}

\begin{cor}\cite[Corollary 2.5]{BMT}, \cite[Theorem 1.16]{My} With the assumptions and notations made in Theorem \ref{BMT}, we have the equivalences of categories
\[\Ext_A^i(T,-):\X_i\leftrightarrows\Y_i\cap\K^\perp:\Tor_i^E(Y,-),\]
where $\X_i=\{X\in\Mod{A}\mid\Ext_A^j(T,X)=0\hbox{ for all }j\geq0, j\neq i\}$ and $\Y_i=\{Y\in\Mod{E}\mid\Tor_j^B(T,Y)=0\hbox{ for all }j\geq0,j\neq i\}$. 
%
%Moreover, if $T$ is classical tilting then $\K^\perp=\Der{E}$.
\end{cor}

\begin{proof}
In order to apply Corollary \ref{prop:counter-equiv-general} we note that in our case $E=\End_A(T)$  is isomorphic to the endomorphism ring of the projective resolution $U$ of $T$ in the category $\Der{A}$. Moreover by definitions of  $\X_i$ and $\Y_i$ we deduce that $\RHom_A(T,X[i])$ and $Y[-i]\dotimes_BT$ are both concentrated in degree $0$, for all $X\in\X_i$ and all $Y\in\Y_i$.

Therefore, we have
\[\RHom_A(T,X[i])\cong\Ho^i(\RHom_A(T,X))\cong\Ext_A^i(T,X)\] and
\[Y[-i]\dotimes_BT\cong\Ho^{-i}(Y\dotimes_BT)\cong\Tor_i^E(T,Y),\] so the proof is complete.
\end{proof}

\subsection*{Acknowledgement} This research is supported by a grant of the Ministry of Research, Innovation and Digitization, CNCS/CCCDI--UEFISCDI, project number PN-III-P4-ID-PCE-2020-0454, within PNCDI III.

\end{document}